\documentclass{amsart}
 
\usepackage{amssymb,latexsym}

\theoremstyle{plain}
\newtheorem{theorem}{Theorem}

\newtheorem{lemma}{Lemma}
\newtheorem{proposition}{Proposition}

\begin{document}
\title[Relational Constraints and Classes of Functions]
{On Closed Sets of Relational Constraints and
 Classes of Functions Closed under
 Variable Substitutions}
\author{Miguel Couceiro}
\address{Department of Mathematics, Statistics and Philosophy\\
University of Tampere\\ 
Kalevantie 4, 33014 Tampere, Finland}
\email{Miguel.Couceiro@uta.fi}
\author{Stephan Foldes}
\address{Institute of Mathematics, Tampere University of Technology\\ 
PL553, 33101 Tampere, Finland}
\email{sf@tut.fi}
\date{Final version $01$-$2005$}

\bigskip

\begin{abstract} Pippenger's Galois theory of finite functions and relational constraints is extended
to the infinite case. The functions involved are functions of several variables on a set $A$ and
taking values in a possibly different set $B$, where any or both of $A$ and $B$ may be finite or infinite.
\end{abstract}

\maketitle

\section{Basic Concepts and Terminology}

In \cite{G} Geiger determined, by explicit closure conditions,
 the closed classes of endofunctions of several variables (operations) and the closed classes of 
relations (predicates)
on a finite set $A$. These two dual closure systems are related in a Galois connection given by 
 the "preservation" relation between endofunctions and relations.
This Galois theory was also developed independently by Bodnarchuk, Kalu\v{z}nin, Kotov and Romov in 
\cite{BK}. Removing the finiteness restriction on the underlying set, in \cite{Sz} Szab\' o 
characterized the closed classes of endofunctions and closed sets of relations on arbritrary sets.
These characterizations involve a local closure property as well as closure under a general scheme
of combining families of relations into a new relation, properly extending the schemes described by Geiger
in the case of finite sets. Different approaches and formulations, as well as variant Galois theories 
were developed by P\" oschel in \cite{Po2}, \cite{Po3}, and \cite{PK} in the case of finite sets 
(see also \cite{R} and \cite{B} for further extensions).

There are many natural classes of functions that can not be defined by preservation of a single relation 
(or preservation of each member of a family of relations),
e.g. monotone decreasing functions on an ordered set, or Boolean functions
 whose Zhegalkin polynomial has degree at most $m\geq 0$.
However such classes can often be described as consisting of those functions that "transform" one 
relation to another relation. Also, many natural classes are not classes of endofunctions,
the sets in which the function variables are interpreted being different from the codomain of 
function values, e.g. rank functions of matroids. In the case of finite sets a theory of such 
functions of several variables, defined as functions from a cartesian product $A_1\times \ldots \times A_n$
of finite sets to a finite set $B$, was developed by P\" oschel in \cite{Po1}:
relations as ordinarily understood are replaced by tuples of relations, 
then the notion of preservation of relations
is naturally extended to such multisorted functions and relational tuples,
and the closed classes of functions and relational tuples are determined with respect to the arising 
Galois connection.
Still in the case of finite sets, in \cite{Pi2} Pippenger developed a particular Galois theory 
for functions $A^n\rightarrow B$, where the dual object role of relations is replaced 
 by ordered pairs of relations called "constraints".
In this paper we extend this latter theory by removing the finiteness restriction.

The functions of several variables we consider in this paper are defined on arbitrary sets, 
not necessarily finite, 
taking values in another, possibly different and possibly infinite set.
 The relations and relational constraints that we consider are also defined on arbitrary,
 not necessarily finite sets.
Positive integers are thought of as ordinals according to the von Neumann conception, 
i.e. each ordinal is just the set of lesser ordinals. Thus,
 for a positive integer $n$ and a set $A$, the $n$-\emph{tuples} in $A^n$ are formally maps from $\{0,\ldots ,n-1\}$ to $A$.
The notation $(a_t\mid t\in n)$ means the $n$-tuple mapping $t$ to $a_t$ for each $t\in n$. 
The notation $(b^1\ldots b^n)$ means the $n$-tuple mapping $t$ to $b^{t+1}$ for each $t\in n$.
 A map (function)
is always thought of as having a specific domain, codomain and graph. 
We need this formalism in order to streamline certain definitions and arguments in later Sections of this paper.

Consider arbitrary non-empty sets $A$ and $B$.

 A \emph{$B$-valued function of several variables on $A$} (or simply, \emph{$B$-valued function on $A$})
is a map $f: A^n \rightarrow B$,
 where the \emph{arity} $n$ is a positive integer.
 Thus the set of all $B$-valued functions on $A$ is $\cup _{n\geq 1}B^{A^n}$.
We also use the term \emph{class} for a set of functions.
If $A=B$, then the $B$-valued functions on $A$ are called \emph{operations on $A$}.
For a fixed arity $n$, the $n$ different \emph{projection maps} 
${\bf a}=(a_t\mid t\in n)\mapsto a_i$, $i\in n$, are 
also called \emph{variables}.

If $l$ is a map from $n$ to $m$ then the $m$-ary function $g$ defined by 
\begin{displaymath}
g({\bf a})=f({\bf a}\circ l)
\end{displaymath}
for every $m$-tuple ${\bf a}\in A^m$, is said to be obtained from the $n$-ary function $f$ by \emph{simple variable substitution}.
Note that this subsumes cylindrification (addition of inessential variables), permutation of variables and
diagonalization (identification of variables), see e.g. \cite{M}, \cite{Pi1}, and \cite{Po1}.
 A class $\mathcal{K}$ of functions of several variables is said to be \emph{closed under simple variable substitutions}
if each function obtained from a function $f$ in $\mathcal{K}$ by simple variable substitution is also in $\mathcal{K}$.
Variable substitution plays a significant role in a number of studies of function classes and 
class definability (see e.g. \cite{WW, W, EFHH, F, Pi2, Z}).

For a positive integer $m$, an $m$\emph{-ary relation on $A$} is a subset $R$ of $A^m$.
For an $m$-tuple ${\bf a}$ we write $R({\bf a})$ if ${\bf a}\in R$.
An $m\times n$ matrix $M$ with entries in $A$ is thought of as an $n$-tuple of $m$-tuples,
$M=({\bf a}^1\ldots {\bf a}^n)$. The $m$-tuples ${\bf a}^1,\ldots ,{\bf a}^n$ are called  
\emph{columns} of $M$.  For $i\in m$,  the 
$n$-tuple $({\bf a}^1(i)\ldots {\bf a}^n(i))$ is called \emph{row} $i$ of $M$ .  
For a matrix $M$ with entries in $A$, we write $M\prec R$ if all columns of $M$ are in $R$.
 For an $n$-ary function $f\in B^{A^n}$ and an $m\times n$ matrix $M=({\bf a}^1\ldots {\bf a}^n)$,
 we denote by $fM$ the $m$-tuple $(f({\bf a}^1(i)\ldots {\bf a}^n(i))\mid i\in m)$ in $B^m$. 
Also, we denote by $fR$ the $m$-ary relation on $B$ given by 
 \begin{displaymath}
 fR=\{fM: \textrm{ $m\times n$ matrix } M\prec R\}. 
\end{displaymath}

\section{Classes of Functions of Several Variables Definable by Relational Constraints}

Consider arbitrary non-empty sets $A$ and $B$.
An $m$-ary \emph{$A$-to-$B$ relational constraint} (or simply, $m$-ary \emph{constraint},
 when the underlying sets are understood from the context)
  is an ordered pair $(R,S)$ where $R\subseteq A^m$ and $S\subseteq B^m$. 
The relations $R$ and $S$ are called \emph{antecedent} and \emph{consequent}, respectively,
 of the constraint.
 A function of several variables 
 $f:A^n\rightarrow B$, $n\geq 1$, is said to \emph{satisfy} an $m$-ary $A$-to-$B$ constraint $(R,S)$ 
if $fR\subseteq S$. For general background see \cite{Pi2}.

A class $\mathcal{K}\subseteq \cup _{n\geq 1}B^{A^n}$ of $B$-valued functions on $A$ is said 
to be \emph{definable} by a set $\mathcal{S}$ of $A$-to-$B$ constraints, if $\mathcal{K}$
 is the class of all functions which satisfy every member of $\mathcal{S}$.

A class $\mathcal{K}\subseteq \cup _{n\geq 1}B^{A^n}$ of $B$-valued functions on $A$ is said 
to be \emph{locally closed} if for every $B$-valued function $f$ of several variables on $A$ the following holds: 
if every restriction of $f$ to a finite subset of its domain $A^n$ coincides with a restriction of 
some member of $\mathcal{K}$, then $f$ belongs to $\mathcal{K}$.

\begin{theorem}
Consider arbitrary non-empty sets $A$ and $B$.
 For any class of functions $\mathcal{K}\subseteq \cup _{n\geq 1}B^{A^n}$ 
the following conditions are equivalent:
\begin{itemize}
\item[(i)] $\mathcal{K}$ is locally closed and it is closed under simple variable substitutions;
\item[(ii)] $\mathcal{K}$ is definable by some set of $A$-to-$B$ constraints.
\end{itemize}
\end{theorem}

\begin{proof}
$(ii)\Rightarrow (i)$: As observed in the finite case by Pippenger in \cite{Pi2}, it is easy to see, also in general,
 that if a function $f$ satisfies a constraint $(R,S)$ then every   
function obtained from $f$ by simple variable substitution also satisfies $(R,S)$. Thus,
any function class $\mathcal{K}$ definable by a set of constraints is closed under 
simple variable substitutions.

 To show that $\mathcal{K}$ is locally closed, consider $f\not\in \mathcal{K}$ 
and let $(R,S)$ be a $A$-to-$B$ constraint that is not satisfied by $f$ but
  satisfied by every function $g$ in $\mathcal{K}$. Thus for some matrix $M\prec R$, 
 $fM\not\in S$ but $gM\in S$ for every $g\in \mathcal{K}$.
 So there is a finite restriction of $f$, namely its restriction to the set of rows of $M$,
 which does not coincide with that of any member of $\mathcal{K}$.

$(i)\Rightarrow (ii)$: We need to show that, for every function $g$ not in $\mathcal{K}$,
there is a $A$-to-$B$ constraint $(R,S)$ such that:
\begin{itemize}
\item[a)] every $f$ in $\mathcal{K}$ satisfies $(R,S)$
\item[b)] $g$ does not satisfy $(R,S)$
\end{itemize}

The case $\mathcal{K}=\emptyset $ being trivial, assume that $\mathcal{K}$ is non-empty.
Suppose that $g$ is $n$-ary. Since $g\not\in \mathcal{K}$,
 there is a finite restriction $g_F$ of $g$ to a finite subset $F\subseteq A^n$
 such that $g_F$ disagrees with every function in $\mathcal{K}$ restricted to $F$. Clearly, $F$ is non empty.
So let $M$ be a $\mid F\mid \times n$ matrix whose rows are the various $n$-tuples in $F$. 
Following Geiger's strategy, also used by Pippenger, define $R$ to be the set of columns of $M$ and
 let $S=\{fM:f\in \mathcal{K}, f$ $n$-ary$\}$.
It is clear from the above construction that $(R,S)$ is an $A$-to-$B$ constraint, and,
 since $\mathcal{K}$ is closed under simple variable substitutions, 
every function in $\mathcal{K}$ satisfies $(R,S)$.
 Also, $g_F$ does not satisfy $(R,S)$, therefore $g$ does not satisfy $(R,S)$ either. 
Thus, conditions a) and b) hold for the constraint $(R,S)$. 
\end{proof}

This generalizes the characterization of closed classes of functions given by Pippenger in \cite{Pi2} 
 by allowing both finite and infinite underlying sets.

\section{Sets of Relational Constraints Characterized by Functions}

The following constructions on maps will be needed.

For maps $f:A\rightarrow B$ and $g:C\rightarrow D$, the composition $g\circ f$ is defined only if
$B=C$. Removing this restriction we define the \emph{concatenation} of $f$ and $g$, denoted simply
 $gf$, to be the map with domain $f^{-1}[B\cap C]$ and codomain $D$ given by 
$(gf)(a)=g(f(a))$ for all $a\in f^{-1}[B\cap C]$. Clearly, if $B=C$ then $gf=g\circ f$, thus concatenation
subsumes and extends functional composition. Concatenation is associative, i.e. for any maps $f$, 
$g$, $h$ we have $h(gf)=(hg)f$.

Given a family $(g_i)_{i\in I}$ of maps, $g_i:A_i\rightarrow B_i$ such that $A_i\cap A_j=\emptyset $
whenever $i\not=j$, we call (\emph{piecewise}) \emph{sum of the family} $(g_i)_{i\in I}$, denoted 
${\Sigma }_{i\in I}g_i$, the map from ${\cup }_{i\in I}A_i$ to ${\cup }_{i\in I}B_i$ whose restriction
to each $A_i$ agrees with $g_i$. If $I$ is a two-element set, say $I=\{1,2\}$, then we write $g_1+g_2$.
Clearly, this operation is associative and commutative.

The operations of concatenation and summation are linked by distributivity, i.e. for any family $(g_i)_{i\in I}$ of maps on disjoint domains
and any map $f$
 \begin{displaymath}
({\Sigma }_{i\in I}g_i)f={\Sigma }_{i\in I}(g_if) \qquad \textrm{ and } \qquad 
f({\Sigma }_{i\in I}g_i)={\Sigma }_{i\in I}(fg_i). 
\end{displaymath} 
In particular, if $g$ and $g'$ are maps with disjoint domains, then 
\begin{displaymath}
(g+g')f=(gf)+(g'f) \qquad \textrm{ and } \qquad 
f(g+g')=(fg)+(fg'). 
\end{displaymath}

Let $g_1,\ldots ,g_n$ be maps from $A$ to $B$. 
The $n$-tuple $(g_1\ldots g_n)$ determines a \emph{vector-valued map} $g:A\rightarrow B^n$, given by
$g(a)=(g_1(a)\ldots g_n(a))$ for every $a\in A$. If $f$ is an $n$-ary $C$-valued function on $B$ then 
the composition $f\circ g$ is a map from $A$ to $C$, it is traditionally denoted by $f(g_1\ldots g_n)$ 
and called the \emph{composition of $f$ with $g_1,\ldots ,g_n$}. Suppose now that $A\cap A'=\emptyset $ and    
${g'}_1,\ldots ,{g'}_n$ are maps from $A'$ to $B$. Letting $g$ and $g'$ be the vector-valued maps 
determined by $(g_1\ldots g_n)$ and $({g'}_1\ldots {g'}_n)$, respectively, we have that
$f(g+g')=(fg)+(fg')$, i.e.  
\begin{displaymath}
f((g_1+{g'}_1)\ldots (g_n+{g'}_n))=f(g_1\ldots g_n)+f({g'}_1\ldots {g'}_n).   
\end{displaymath} 

For $B\subseteq A$, ${\iota }_{AB}$ denotes the canonical injection (inclusion map) from $B$ to $A$.
Thus the restriction $f\mid _B$ of any map $f:A\rightarrow C$
 to the subset $B$ is given by $f\mid _B=f{\iota }_{AB}$.

To discuss closed sets of constraints we need the following concepts.

We denote the binary equality relation on a set $A$ by $=_A$. The \emph{binary $A$-to-$B$ equality constraint} is
$(=_A,=_B)$.
A constraint $(R,S)$ is called the \emph{empty constraint} if both antecedent and consequent are empty.
For every $m\geq 1$, the constraints $(A^m,B^m)$ are said to be \emph{trivial}.
Note that every $B$-valued function on $A$ satisfies each of these constraints.

A constraint $(R,S)$ is said to be obtained from $(R_0,S_0)$ by \emph{restricting the antecedent} 
if $R\subseteq R_0$ and $S=S_0$.
Similarly, a constraint $(R,S)$ is said to be obtained from $(R_0,S_0)$ by \emph{extending the consequent} 
if $S\supseteq S_0$ and $R=R_0$. 
If a  constraint $(R,S)$ is obtained from $(R_0,S_0)$ by restricting the antecedent or extending 
the consequent or a combination of the two (i.e. $R\subseteq R_0$ and $S\supseteq S_0$)
we say that $(R,S)$ is a \emph{relaxation} of $(R_0,S_0)$.
Given a non-empty family of constraints $(R,S_j)_{j\in J}$ of the same arity (and antecedent), 
 the constraint $(R,\cap _{j\in J} S_j)$ is said to be obtained from $(R,S_j)_{j\in J}$ 
by \emph{intersecting consequents}.

The above operations were introduced by Pippenger in \cite{Pi2}
 in the context of finite sets, together with the notion of "simple minors".
We propose a minor formation concept which extends and subsumes these operations.
This concept is closely related to the construction of relations via the "formula schemes"
of Szab\' o (see \cite{Sz}) and the "general superpositions" of P\" oschel  
(see e.g. \cite{Po2}, \cite{Po3}). We shall discuss this relationship in Section 5.

Let $m$ and $n$ be positive integers (viewed as ordinals, i.e., $m=\{0,\ldots ,m-1\}$).
Let $h:n\rightarrow m\cup V$ where
$V$ is an arbitrary set of symbols disjoint from the ordinals 
called "\emph{existentially quantified indeterminate indices}", or simply \emph{indeterminates},
 and $\sigma :V\rightarrow A$ any map called a \emph{Skolem map}.
Then each $m$-tuple ${\bf a}\in A^m$, being a map ${\bf a}:m\rightarrow A$,
gives rise to an $n$-tuple $({\bf a}+\sigma )h\in A^n$.

Let $H=(h_j)_{j\in J}$ be a non-empty family of maps $h_j:n_j\rightarrow m\cup V$, where each $n_j$ is a positive integer
(recall $n_j=\{0,\ldots ,n_j-1\}$). Then $H$ is called a \emph{minor formation scheme} with \emph{target} $m$,
 \emph{indeterminate set} $V$ and \emph{source family} $(n_j)_{j\in J}$.
Let $(R_j)_{j\in J}$ be a family of relations (of various arities) on the same set $A$, each $R_j$ of arity $n_j$,
and let $R$ be an $m$-ary relation on $A$. We say that $R$ is a  
 \emph{restrictive conjunctive minor} of the family $(R_j)_{j\in J}$ \emph{via $H$},
or simply a \emph{restrictive conjunctive  minor} of the family $(R_j)_{j\in J}$, if 
for every $m$-tuple $\bf a$ in $A^m$, the condition $R({\bf a})$  
implies that there is a Skolem map $\sigma :V\rightarrow A$ such that, for all $j$ in $J$, we have
$R_j[({\bf a}+\sigma )h_j]$.
On the other hand, if for every $m$-tuple $\bf a$ in $A^m$, the condition $R({\bf a})$ holds whenever
there is a Skolem map $\sigma :V\rightarrow A$ such that, for all $j$ in $J$, we have 
$R_j[({\bf a}+\sigma )h_j]$, then
we say that $R$ is an \emph{extensive conjunctive minor} of the family $(R_j)_{j\in J}$ \emph{via $H$},
or simply an \emph{extensive conjunctive minor} of the family $(R_j)_{j\in J}$.
If $R$ is both a restrictive conjunctive minor and 
an extensive conjunctive minor of the family $(R_j)_{j\in J}$ via $H$, 
then $R$ is said to be a \emph{tight conjunctive minor} of the family $(R_j)_{j\in J}$ \emph{via $H$},
or \emph{tight conjunctive minor} of the family. Note that given a scheme $H$ and a family $(R_j)_{j\in J}$,
 there is a unique tight conjunctive minor of the family $(R_j)_{j\in J}$ via $H$.

If $(R_j,S_j)_{j\in J}$ is a family of $A$-to-$B$ constraints (of various arities) and $(R,S)$ is an 
$A$-to-$B$ constraint such that for a scheme $H$
\begin{itemize}
\item[(i)] $R$ is a restrictive conjunctive minor of $(R_j)_{j\in J}$ via $H$, 
\item[(ii)] $S$ is an extensive conjunctive minor of $(S_j)_{j\in J}$ via $H$, 
\end{itemize}
then $(R,S)$ is said to be a \emph{conjunctive minor} of the family $(R_j,S_j)_{j\in J}$ \emph{via $H$},
or simply a \emph{conjunctive minor} of the family of constraints.

If both $R$ and $S$ are tight conjunctive minors of the respective families via $H$, the constraint   
$(R,S)$ is said to be a \emph{tight conjunctive minor} of the family $(R_j,S_j)_{j\in J}$ \emph{via $H$}, 
or simply a \emph{tight conjunctive minor} of the family of constraints.
Note that given a scheme $H$ and a family $(R_j,S_j)_{j\in J}$, 
 there is a unique tight conjunctive minor of the family via the scheme $H$.

An important particular case of tight conjunctive minors is when the minor formation scheme $H=(h_j)_{j\in J}$ 
and the family $(R_j,S_j)_{j\in J}$  are indexed by a singleton $J=\{0\}$. 
In this case, a tight conjunctive minor $(R,S)$ of a family containing a single constraint $(R_0,S_0)$ 
is called a \emph{simple minor} of $(R_0,S_0)$ according to the concept introduced by Pippenger in 
\cite{Pi2}.

\begin{lemma}
Let $(R,S)$ be a conjunctive minor of a non-empty family $(R_j,S_j)_{j\in J}$ of $A$-to-$B$ constraints.
If $f:A^n\rightarrow B$ satisfies every $(R_j,S_j)$ then $f$ satisfies $(R,S)$.
\end{lemma}

\begin{proof} Let $(R,S)$ be an $m$-ary conjunctive minor of the family $(R_j,S_j)_{j\in J}$ via the scheme  
$H=(h_j)_{j\in J}$, $h_j:n_j\rightarrow m\cup V$. Let $M=({\bf a}^1\ldots {\bf a}^n)$ 
be an $m\times n$ matrix with columns in $R$. 
We need to prove that the $m$-tuple $fM$ belongs to $S$. 
Note that the $m$-tuple $fM$, being a map defined on $m$, is in fact the composition of $f$ with the 
 $m$-tuples ${\bf a}^1,\ldots ,{\bf a}^n$, i.e. $fM=f({\bf a}^1\ldots {\bf a}^n)$.   
Since $R$ is a restrictive conjunctive minor of $(R_j)_{j\in J}$ via $H=(h_j)_{j\in J}$,
 there are Skolem maps ${\sigma }_i:V\rightarrow A$, $1\leq i\leq n$, such that for every $j$ in $J$,
 for the matrix $M_j=(({\bf a}^1+\sigma _1)h_j\ldots ({\bf a}^n+\sigma _n)h_j)$ 
we have $M_j\prec R_j$.

Since $S$ is an extensive conjunctive minor of $(S_j)_{j\in J}$ via the same scheme $H=(h_j)_{j\in J}$,
to prove that $fM$ is in $S$, 
it suffices to give a Skolem map $\sigma :V\rightarrow B$ such that, for all $j$ in $J$, the $n_j$-tuple
$(fM+\sigma )h_j$ belongs to $S_j$.
Let $\sigma =f({\sigma }_1\ldots {\sigma }_n)$.
By the rules discussed at the begining of this Section, we have that for each $j$ in $J$,
\begin{displaymath}
(fM+\sigma )h_j=[f({\bf a}^1\ldots {\bf a}^n)+f({\sigma }_1\ldots {\sigma }_n)]h_j= 
[f(({\bf a}^1+{\sigma }_1)\ldots ({\bf a}^n+{\sigma }_n))]h_j= \\
\end{displaymath} 
\begin{displaymath}
=f[({\bf a}^1+{\sigma }_1)h_j\ldots ({\bf a}^n+{\sigma }_n)h_j]=fM_j
\end{displaymath} 
Since $f$ satisfies $(R_j,S_j)$, we have $fM_j\in S_j$.
\end{proof}

We say that a class $\mathcal{T}$ of relational constraints is 
\emph{closed under formation of conjunctive minors}
if whenever every member of the non-empty family $(R_j,S_j)_{j\in J}$ of constraints
 is in $\mathcal{T}$, all conjunctive minors of the family $(R_j,S_j)_{j\in J}$ are also in $\mathcal{T}$.

The formation of conjunctive minors subsumes the formation of simple minors as well as the operations of
restricting antecedents, extending consequents and intersecting consequents.
 Simple minors in turn subsume permutation, identification,
projection and addition of dummy arguments (see Pippenger \cite{Pi2}).

In analogy with locally closed function classes,
we say that a set $\mathcal{T}$ of relational constraints is \emph{locally closed} if 
for every $A$-to-$B$ constraint $(R,S)$ the following holds: 
if every relaxation of $(R,S)$ with finite antecedent coincides with 
some member of $\mathcal{T}$, then $(R,S)$ belongs to $\mathcal{T}$.

A set $\mathcal{T}$ of $A$-to-$B$ constraints is said to be \emph{characterized} by a set $\mathcal{F}$
of $B$-valued functions on $A$ if $\mathcal{T}$ is the set of all those constraints
which are satisfied by every member of $\mathcal{F}$.

\begin{theorem}
Consider arbitrary non-empty sets $A$ and $B$.
Let $\mathcal{T}$ be a set of $A$-to-$B$ relational constraints. Then the following are equivalent:
\begin{itemize}
\item[(i)]$\mathcal{T}$ is locally closed and contains the binary equality constraint,
 the empty constraint, and it is closed under formation of conjunctive minors;
\item[(ii)]$\mathcal{T}$ is characterized by some set of $B$-valued functions on $A$.
\end{itemize}
\end{theorem}

\begin{proof} First we prove the implication $(ii)\Rightarrow (i)$. 
It is clear that every function on $A$ to $B$ 
satisfies the empty and the equality constraints. 
It follows from Lemma 1 that if a function satisfies 
a non-empty family $(R_j,S_j)_{j\in J}$ of constraints then it satisfies every conjunctive minor of the family.
Thus, to prove the implication \emph{(ii) $\Rightarrow $(i)} we only need to 
show that $\mathcal{T}$ is locally closed.
 For that, let $(R,S)$ be an $m$-ary constraint not in $\mathcal{T}$.  
 By $(ii)$, there is an $n$-ary function $f$ satisfying every constraint in 
$\mathcal{T}$ which does not satisfy $(R,S)$. 
Thus, for an $m\times n$ matrix $M\prec R$, $fM\not\in S$. It is easy to see that the constraint $(F,S)$,
 where $F$ is the set of columns of $M$, is a relaxation of $(R,S)$ with finite antecedent
such that $(F,S)\not\in \mathcal{T}$. This completes the proof of implication  $(ii)\Rightarrow (i)$.

To prove the implication $(i)\Rightarrow (ii)$, we need to extend the concepts of relation 
and constraint to infinite arities.
Function arities remain finite. These extended definitions have no bearing on the Theorem itself, 
but are needed only as tools in its proof.

For any non-zero, possibly infinite, ordinal $m$, an \emph{$m$-tuple} is a map defined on $m$.
(An ordinal $m$ is the set of lesser ordinals.) 
Relation and constraint arities are thus allowed to be arbitrary non-zero, possibly infinite,
 ordinals $m,n,\mu $ etc. In minor formation schemes, the target $m$ and the members $n_j$ of the source family
 are also allowed to be arbitrary non-zero, possibly infinite ordinals. For relations, 
we shall use the term \emph{restrictive conjunctive $\infty $-minor} 
(\emph{extensive conjunctive $\infty $-minor}) to indicate a restrictive conjunctive minor
(extensive conjunctive minor, respectively) via a scheme whose target and
source ordinals may be infinite or finite. Similarly, for constraints
we shall use the term \emph{conjunctive $\infty $-minor} 
(\emph{simple $\infty $-minor}) to indicate a conjunctive minor
(simple minor, respectively) via a scheme whose target and
source ordinals may be infinite or finite. 
Thus in the sequel the use of the term "minor" without the prefix "$\infty $-" 
continues to mean the respective minor via 
a scheme whose target and
source ordinals are all finite.
 Matrices can also have infinitely many rows but only finitely many columns:
an $m\times n$ matrix $M$, where $n$ is finite but $m$ could be infinite, 
is an $n$-tuple of $m$-tuples $M=({\bf a}^1\ldots {\bf a}^n)$.

In order to discuss the formation of repeated $\infty $-minors, we need the following definition.
Let $H=(h_j)_{j\in J}$ be a minor formation scheme with target $m$, indeterminate set $V$ and source family 
$(n_j)_{j\in J}$, and, for each $j\in J$, let $H_j=(h_{j}^i)_{j\in J, i\in I_j}$ be a scheme
  with target $n_j$, indeterminate set $V_j$ and source family 
$(n_{j}^i)_{ i\in I_j}$.
Assume that $V$ is disjoint from the $V_j$'s, and for distinct $j$'s the $V_j$'s are also pairwise disjoint.
Then the \emph{composite scheme} $H(H_j : j\in J)$ is the scheme $K=(k_{j}^i)_{j\in J, i\in I_j}$ 
defined as follows:
\begin{itemize}
\item[(i)] the target of $K$ is the target $m$ of $H$, 
\item[(ii)] the source family of $K$ is $(n_{j}^i)_{j\in J, i\in I_j}$, 
\item[(iii)] the indeterminate set of $K$ is $U=V\cup ({\cup }_{j\in J}V_j)$,
\item[(iv)] $k_{j}^i:n_{j}^i\rightarrow m\cup U$ is defined by
\begin{displaymath}
k_{j}^i=(h_j+\iota _{UV_j})h_{j}^i  
 \end{displaymath} 
  where $\iota _{UV_j}$ is the canonical injection (inclusion map) from $V_j$ to $U$.
\end{itemize}

\emph{Claim 1.} If $(R,S)$ is a conjunctive $\infty $-minor of a non-empty family $(R_j,S_j)_{j\in J}$ 
of $A$-to-$B$ constraints 
via the scheme $H$, and, for each $j\in J$, $(R_j,S_j)$ is a conjunctive $\infty $-minor of a non-empty family 
$(R_{j}^i,S_{j}^i)_{i\in I_j}$ via the scheme $H_j$,
then $(R,S)$ is a conjunctive $\infty $-minor of the non-empty family $(R_{j}^i,S_{j}^i)_{j\in J,i\in I_j}$
 via the composite scheme $K=H(H_j : j\in J)$.
  
{\sl Proof of Claim 1.} First, we need to see that $R$ is a restrictive conjunctive $\infty $-minor of the family
 $(R_{j}^i)_{j\in J,i\in I_j}$ via $K$.
Let ${\bf a}$ be an $m$-tuple in $R$.
This implies that there is a Skolem map $\sigma :V\rightarrow A$ such that for all $j$ in $J$,
we have $({\bf a}+\sigma )h_j\in R_j$.
In turn this implies that for every $j$ in $J$ 
there are Skolem maps ${\sigma }_j:V_j\rightarrow A$ such that for every $i$ in $I_j$, the $n_{j}^i$-tuple 
$[({\bf a}+\sigma )h_j+\sigma _j]h_{j}^i$ is in $R_{j}^i$.
Define the Skolem map $\tau :U\rightarrow A$ by $\tau =\sigma +{\Sigma }_{l\in J}{\sigma }_l$.
Then for every $j\in J$ and $i\in I_j$, we have $({\bf a}+\tau ){k_{j}^i}\in R_{j}^i$ because 
\begin{displaymath}
  ({\bf a}+\tau ){k_{j}^i}=({\bf a}+\sigma +{\Sigma }_{l\in J}{\sigma }_l)(h_j+\iota _{UV_j})h_{j}^i=
 \end{displaymath} 
\begin{displaymath}
  =[({\bf a}+\sigma )h_j+({\Sigma }_{l\in J}{\sigma }_l)h_j+({\bf a}+\sigma )
\iota _{UV_j}+({\Sigma }_{l\in J}{\sigma }_l)\iota _{UV_j}]h_{j}^i
=[({\bf a}+\sigma )h_j+\sigma _j]h_{j}^i \qquad (1)
 \end{displaymath} 
and this $n_{j}^i$-tuple is in $R_{j}^i$.

Second, we need to see that $S$ is an extensive conjunctive $\infty $-minor of 
the family $(S_{j}^i)_{j\in J,i\in I_j}$ via $K$. 
Take an $m$-tuple ${\bf b}\in B^m$ and assume that there is a Skolem map $\tau :U\rightarrow B$
such that for every $j\in J$ and $i\in I_j$, the $n_{j}^i$-tuple $({\bf b}+\tau ){k_{j}^i}$ is in $S_{j}^i$.
We need to show that $\bf b$ is in $S$. 
Define the Skolem maps $\sigma :V\rightarrow B$ and 
${\sigma }_j:V_j\rightarrow B$ for every $j\in J$, by restriction of $\tau $,
i.e. $\tau =\sigma +{\Sigma }_{l\in J}{\sigma }_l$.
Similarly to $(1)$, 
\begin{displaymath}
  ({\bf b}+\tau ){k_{j}^i}
=[({\bf b}+\sigma )h_j+\sigma _j]h_{j}^i. 
 \end{displaymath} 
Since $S_j$ is an extensive conjunctive $\infty $-minor of the family $(S_{j}^i)_{j\in J,i\in I_j}$ 
via the scheme $H_j$ we have 
$({\bf b}+\sigma )h_j\in S_j$. As the condition $({\bf b}+\sigma )h_j\in S_j$ holds for all $j$ in $J$ and 
$S$ is an extensive conjunctive $\infty $-minor of the family $(S_{j})_{j\in J}$ via $H$,
 we have that $\bf b$ is in $S$,   
 which completes the proof of Claim 1.

For a set $\mathcal{T}$ of $A$-to-$B$ constraints, we denote by ${\mathcal{T}}^{\infty }$ 
the set of those constraints
which are conjunctive $\infty $-minors of families of members of $\mathcal{T}$. 
This set ${\mathcal{T}}^{\infty }$ is the smallest set
of constraints containing $\mathcal{T}$ which is closed under formation of conjunctive $\infty $-minors and 
it is called the \emph{conjunctive $\infty $-minor closure} of $\mathcal{T}$.  
In the sequel, we shall make use of the following fact:

\emph{Fact 1.} Let $\mathcal{T}$ be a set of finitary $A$-to-$B$ constraints and 
let ${\mathcal{T}}^{\infty }$ be its conjunctive $\infty $-minor closure.
If $\mathcal{T}$ is closed under formation of conjunctive minors,
 then $\mathcal{T}$ is the set of all finitary constraints belonging to ${\mathcal{T}}^{\infty }$.

\emph{Claim 2.} Let $\mathcal{T}$ be a locally closed set of finitary $A$-to-$B$ constraints containing 
the binary equality constraint, the empty constraint,
 and closed under formation of conjunctive minors, and let ${\mathcal{T}}^{\infty }$ 
be its $\infty $-minor closure. 
Let $(R,S)$ be a finitary $A$-to-$B$ constraint not in $\mathcal{T}$.
Then there is a $B$-valued function $g$ on $A$ such that   
\begin{itemize}
\item[1)] $g$ satisfies every constraint in ${\mathcal{T}}^{\infty }$ 
\item[2)] $g$ does not satisfy $(R,S)$
\end{itemize}

{\sl Proof of Claim 2.} We shall construct a function $g$ which satisfies all constraints in 
${\mathcal{T}}^{\infty }$ but $g$ does not satisfy $(R,S)$.

Note that, by Fact 1, $(R,S)$ can not be in ${\mathcal{T}}^{\infty }$.
Let $m$ be the arity of $(R,S)$. 
Since $\mathcal{T}$ is locally closed and $(R,S)$ does not belong to $\mathcal{T}$,
 we know that there is a relaxation $(R_1,S_1)$ of $(R,S)$, where $R_1$ is finite, which is not in $\mathcal{T}$. 
Let $n$ be the number of $m$-tuples in $R_1$.
Observe that $S_1\not=B^m$, since the constraint $(A^m,B^m)$ is a simple minor of the 
binary equality constraint, and thus is in $\mathcal{T}$.
Also, $R_1$ is non empty, otherwise $(R_1,S_1)$ would be a relaxation of the empty constraint, and so would belong
to $\mathcal{T}$. Suppose $R_1$ consists of $n$ distinct $m$-tuples ${\bf d}^1,\ldots ,{\bf d}^n$.

 Consider the $m\times n$ matrix $F=({\bf d}^1\ldots {\bf d}^n)$.
Let $M=({\bf a}^1\ldots {\bf a}^n)$ be any matrix whose first $m$ rows are the rows of $F$ 
(i.e. $({\bf a}^1(i)\ldots {\bf a}^n(i))=({\bf d}^1(i)\ldots {\bf d}^n(i))$ for every $i\in m$)
 and whose other rows are
 the remaining distinct $n$-tuples in $A^n$: 
 every $n$-tuple in $A^n$ is a row of $M$, and any repetition of rows can only occur among the first $m$
rows of $M$.
Let $R_M$ be the relation whose elements are the columns of $M$, say of arity $\mu $.
Note that $m\leq \mu $ and that $\mu $ is infinite if and only if $A$ is infinite.
 Let $S_M$ be the $\mu $-ary relation 
consisting of those $\mu $-tuples ${\bf b}=(b_t\mid t\in \mu )$ in $B^{\mu }$ such that $(b_t\mid t\in m)$
belongs to $S_1$.

 Observe that $(R_M,S_M)$ can not belong to ${\mathcal{T}}^{\infty }$, because $(R_1,S_1)$ is a simple $\infty $-minor
of the possibly infinitary constraint $(R_M,S_M)$, and if $(R_M,S_M)\in {\mathcal{T}}^{\infty }$
 we would conclude, from Fact 1,
 that $(R_1,S_1)$ is in $\mathcal{T}$. 
 Also, there must exist a $\mu $-tuple ${\bf s} = (s_t\mid t\in \mu )$ in $B^{\mu }$ such that 
$(s_t\mid t\in m)$ is not in $S_1$, and for which $(R_M,B^{\mu }\setminus \{{\bf s}\})$ 
does not belong to $\mathcal{T}^{\infty }$, otherwise by arbitrary intersections of consequents we would 
conclude that $(R_M,S_M)$ belongs to ${\mathcal{T}}^{\infty }$.

Next we show that if two rows of $M$, say row $i$ and $j$, coincide, then the corresponding components of $\bf s$ also
coincide, $s_i=s_j$. For a contradiction, suppose that rows $i$ and $j$ coincide but $s_i\not=s_j$. 
Consider the $\mu $-ary $A$-to-$B$ constraint $(R^=,S^=)$ defined by 
 \begin{displaymath}
 R^==\{(a_t\mid t\in \mu ): a_i=a_j \} \quad \textrm{ and } \quad 
 S^==\{(b_t\mid t\in \mu ): b_i=b_j \}
\end{displaymath} 
The constraint $(R^=,S^=)$ is a simple $\infty $-minor of the binary equality constraint 
and therefore belongs to ${\mathcal{T}}^{\infty }$.
On the other hand $(R_M,B^{\mu }\setminus \{{\bf s}\})$ is a relaxation of $(R^=,S^=)$ and should also
 belong to ${\mathcal{T}}^{\infty }$,
yielding the intended contradiction.

Observe that the set of rows of $M$ is the set all $n$-tuples of $A^n$. Also, in view of the above, 
we can define an $n$-ary function $g$ by the condition $gM=\bf s$. 
 By definition of $\bf s$, $g$ does not satisfy $(R_M,S_M)$, and so it does not satisfy $(R_1,S_1)$.
So the function $g$ does not satisfy $(R,S)$.

Suppose that there is a $\rho $-ary constraint $(R_0,S_0)\in {\mathcal{T}}^{\infty }$, possibly infinitary,
 which $g$ does not satisfy.
Thus, for some $\rho \times n$ matrix $M_0=({\bf c}^1\ldots {\bf c}^n)$ with columns in $R_0$ we have 
$gM_0\not\in S_0$.
Define $h:\rho  \rightarrow {\mu }$ to be any map such that 
 \begin{displaymath}
({\bf c}^1(i)\ldots {\bf c}^n(i))=(({\bf a}^1h)(i)\ldots ({\bf a}^nh)(i))
\end{displaymath} 
 for every $i\in \rho $,
i.e. row $i$ of $M_0$ is the same as row $h(i)$ of $M$, for each $i\in \rho $.
Let $(R_h,S_h)$ be the ${\mu }$-ary simple $\infty $-minor of $(R_0,S_0)$ via $H=\{h\}$. 
Note that, by Claim 1,
 $(R_h,S_h)$ belongs to ${\mathcal{T}}^{\infty }$.

We claim that $R_M\subseteq R_h$. Any $\mu $-tuple in $R_M$ is a column ${\bf a}^j$ of 
$M=({\bf a}^1\ldots {\bf a}^n)$. To prove that ${\bf a}^j\in R_h$ we need to show that the $\rho $-tuple
${\bf a}^jh$ is in $R_0$. In fact, we have  
\begin{displaymath}
{\bf a}^jh=({\bf a}^jh(i)\mid i \in \rho )=({\bf c}^j(i)\mid i \in \rho )
\end{displaymath} 
and this $\rho $-tuple is in $R_0$.

Next we claim that $B^{\mu }\setminus \{{\bf s}\}\supseteq S_h$, i.e. that ${\bf s}\not\in S_h$. 
For that it is enough to show that ${\bf s}h\not\in S_0$.  
For every $i\in \rho $ we have
\begin{displaymath}
({\bf s}h)(i)=[g({\bf a}^1\ldots {\bf a}^n)h](i)=g[({\bf a}^1h)(i)\ldots ({\bf a}^nh)(i)]=g({\bf c}^1(i)\ldots {\bf c}^n(i))
\end{displaymath} 
Thus ${\bf s}h=gM_0$. Since $gM_0\not\in S_0$ we conclude that ${\bf s}\not\in S_h$. 

 So $(R_M,B^{\mu }\setminus \{{\bf s}\})$ is a relaxation of $(R_h,S_h)$ and 
we conclude that $(R_M,B^{\mu }\setminus \{{\bf s}\})$ is in ${\mathcal{T}}^{\infty }$.
By definition of $\bf s$, this is impossible. Thus we have proved Claim 2.

To see that the implication $(ii)\Rightarrow (i)$ of Theorem $2$ holds, observe that, by Claim 2,
for every constraint $(R,S)$ not in $\mathcal{T}$ there is a function $g$ 
which does not satisfy $(R,S)$ but satisfies every constraint in ${\mathcal{T}}^{\infty }$, and hence 
 satisfies every constraint in $\mathcal{T}$. Thus the set of all these "separating" functions constitutes 
the desired set characterizing $\mathcal{T}$.  
\end{proof}

Theorem $2$ generalizes the characterization of closed classes of constraints given by Pippenger 
in \cite{Pi2}
 by allowing both finite and infinite underlying sets 
and extending the closure conditions on classes of relational constraints (via the broadening of the concept 
of simple minors). The proof of Claim 2, being part of the proof of Theorem 2, 
differs from the analogous constructions of Geiger in \cite{G}
 and of Pippenger in \cite{Pi2} in that the function
$g$ separating the constraint $(R,S)$ is not obtained by successive extensions of partial functions but 
it is defined at once as a total function on $A^n$. 

Conjunctive minors are indeed strictly more general than simple minors:
the fact that the former are not subsumed by the latter in the infinite case is 
illustrated in the following Section.

\section{Comparison of Closures Based on Simple and Conjunctive Minors}

An $n$-ary \emph{$B$-valued partial function on $A$} is 
a function $p: D \rightarrow B$ where $D\subseteq A^n$. The partial function $p$ is said to be 
\emph{finite} if $D$ is a finite set.
If $\mathcal{F}$ is a set of $B$-valued partial functions on $A$, $\mathcal{F}$ is called an 
\emph{extensible family}
if for every $p\in \mathcal{F}$, $p: D \rightarrow B$ where $D\subseteq A^n$, and every 
${\bf y}\in A^n\setminus D$, $\mathcal{F}$ contains an extension $p': D' \rightarrow B$
of $p$ to the domain $D'=D\cup \{{\bf y}\}$.

If $p$ is an $n$-ary $B$-valued partial function on $A$ and $(R,S)$ an $m$-ary $A$-to-$B$ constraint  
we say that $p$ \emph{satisfies} $(R,S)$ if for every $m\times n$ matrix $M\prec R$, 
$M=({\bf a}^1\ldots {\bf a}^n)$,
 and such that every row of $M$ belongs to the domain of $p$, the $m$-tuple 
$(p({\bf a}^1(i)\ldots {\bf a}^n(i))\mid i\in m)$ 
belongs to $S$ .

\begin{proposition} For any sets $A$ and $B$, 
the set of all $A$-to-$B$ constraints satisfied by an extensible family $\mathcal{F}$ of $B$-valued partial 
functions on $A$ is locally closed
 and contains the binary equality constraint, the empty constraint, and it is closed under intersecting consequents
and under taking simple minors. 
\end{proposition}

\begin{proof} The only non-trivial claim is that the set of 
all $A$-to-$B$ constraints satisfied by the extensible family $\mathcal{F}$ of $B$-valued partial 
functions on $A$ is closed under taking simple minors. 
Suppose that every member of the extensible family $\mathcal{F}$ satisfies an $n$-ary constraint
$(R_0,S_0)$. Let $(R,S)$ be an $m$-ary simple minor of $(R_0,S_0)$ via   
$h:n\rightarrow m\cup V$. 

Let $p\in \mathcal{F}$, $p: D \rightarrow B$, $D\subseteq A^t$. We need to show that 
$p$ satisfies $(R,S)$.
Take an $m\times t$ matrix $M=({\bf a}^1\ldots {\bf a}^t)$, $M\prec R$, 
 such that every row of $M$ is in $D$.
We know that ${\bf a}^1,\ldots ,{\bf a}^t\in R$, that is, there are Skolem maps 
${\sigma }_1,\ldots ,\sigma _t:V\rightarrow A$ such that 
$({\bf a}^1+\sigma _1)h,\ldots ({\bf a}^t+\sigma _t)h\in  R_0$.

We claim that $(p({\bf a}^1(i)\ldots {\bf a}^t(i))\mid i\in {m})$ belongs to $S$.
It is enough to show that for some $f: A^t \rightarrow B$ (not necessarily in $\mathcal{F}$)
such that $f\mid _D=p$ we have $f({\bf a}^1\ldots {\bf a}^t)\in S$. 
This latter membership in $S$ is equivalent to the existence of a Skolem map 
${\sigma }:V\rightarrow B$ such that $(f({\bf a}^1\ldots {\bf a}^t)+\sigma )h\in S_0$.   
We shall define such an $f$ and $\sigma $.

Observe that the set $A_0=\{({\sigma }_1(v),\ldots ,\sigma _t(v)):v\in V\cap h[n]\}$, where $h[n]$
is the range of $h$, is finite. By a straightfoward induction based on the definition of an extensible
family, it follows that there is an extension $p'$ of $p$, $p'$ in $\mathcal{F}$, whose domain is 
$D'=D\cup A_0$. Let ${\sigma }:V\rightarrow B$ be any Skolem map such that
$\sigma (v)=p'({\sigma }_1(v),\ldots ,\sigma _t(v))$ for all $v\in V\cap h[n]$.
 Note that every row of the $n\times t$ matrix 
$N=(({\bf a}^1+{\sigma }_1)h\ldots ({\bf a}^t+{\sigma }_t)h)$ is in the domain of $p'$.

Let $f: A^t \rightarrow B$ be any function (not necessarily in $\mathcal{F}$)
such that $f\mid _{D'}=p'$. We show that $(f({\bf a}^1\ldots {\bf a}^t)+\sigma )h\in S_0$.   
Using the rules in Section 3, we have
\begin{displaymath}
\begin{array}{l}
[f({\bf a}^1\ldots {\bf a}^t)+\sigma ]h=[f({\bf a}^1\ldots {\bf a}^t)+f({\sigma }_1\ldots {\sigma }_t)]h 
=[f(({\bf a}^1+{\sigma }_1)\ldots ({\bf a}^t+{\sigma }_t))]h=\\
=f[({\bf a}^1+{\sigma }_1)h\ldots ({\bf a}^t+{\sigma }_t)h]=
(p'[({\bf a}^1+{\sigma }_1)h(j)\ldots ({\bf a}^t+{\sigma }_t)h(j)]\mid j\in {n}) \qquad (2) 
\end{array}
\end{displaymath} 
\begin{displaymath}
\end{displaymath} 
Since $p'\in \mathcal{F}$, $p'$ satisfies $(R_0,S_0)$, and as the rows of $N$ are in the domain of $p'$,
we have that $(2)$ is in $S_0$.
 \end{proof}

Let $A$ and $B$ be sets with different infinite cardinalities, $card(A)>card(B)$. Let $\mathcal{F}$ 
be the set of all injective finite $B$-valued partial functions of several variables on $A$: 
this $\mathcal{F}$ is an extensible family.
Let $\Delta _A$ and $\Delta _B$ be the binary disequality relations on $A$ and $B$, respectively, defined by
 \begin{displaymath}
 \Delta _A=\{(a,b)\in A^2\mid a\not=b \}
\end{displaymath} 
 \begin{displaymath}
 \Delta _B=\{(c,d)\in B^2\mid c\not=d \}
\end{displaymath} 
It is easy to see that every member of $\mathcal{F}$ satisfies $(\Delta _A,\Delta _B$).

Let $V$ be a set of indeterminates equipotent to $A$ and let $\mathcal{V}$ be the set of all two-element subsets of $V$.
Denote $\{\alpha ,\beta \}\in \mathcal{V}$ by $\alpha \beta $ for short.
Take any strict total ordering $<$ on $V$, and for $\{\alpha ,\beta \}\in \mathcal{V}$ define
\begin{displaymath} 
 h_{\alpha \beta }:2\rightarrow 1\cup V     
\end{displaymath} 
 by $h_{\alpha \beta }(0)=min(\alpha ,\beta )$
and $h_{\alpha \beta }(1)=max(\alpha ,\beta )$.
(Actually, $h_{\alpha \beta }$ could be any map $2\rightarrow 1\cup V$ with range $\{\alpha ,\beta \}$.)

Define the family $(R_{\alpha \beta },S_{\alpha \beta })_{\alpha \beta \in \mathcal{V}}$
 of constraints by $(R_{\alpha \beta },S_{\alpha \beta })=(\Delta _A,\Delta _B)$
for all $\alpha \beta \in \mathcal{V}$.

The tight conjunctive minor of the family $(R_{\alpha \beta })_{\alpha \beta \in \mathcal{V}}$
via the scheme $H=(h_{\alpha \beta })_{\alpha \beta \in \mathcal{V}}$ is the full unary relation $A^1$ on $A$,
while the tight conjunctive minor of $(S_{\alpha \beta })_{\alpha \beta \in \mathcal{V}}$ via $H$
is the empty unary relation.
Therefore, $(A^1,\emptyset )$ is the tight conjunctive minor of the the family of constraints 
$(R_{\alpha \beta },S_{\alpha \beta })_{\alpha \beta \in \mathcal{V}}$ via $H$,
but clearly is not satisfied by the members
of $\mathcal{F}$. Thus, in view of Proposition 1, we obtain the following:

\begin{theorem}
 Conjunctive minors subsume simple minors, relaxations and intersections of consequents,
 but there are conjunctive minors
which can not be obtained by any combination of taking simple minors, relaxations or intersections of consequents.
\end{theorem}

In other words, conjunctive minors properly extend the notion of simple minors, and for infinite sets 
Theorem 2 in the previous Section can not be strengthened by replacing "conjunctive minors" with "simple minors",
as it can be in the finite case.

\section{Clones of Functions and Closed Sets of Relations}

In this section we make use of Theorems 1 and 2 in Sections 2 and 3, respectively, to derive variant
characterizations for clones of functions (operations) and closed sets of relations on arbitrary,
not necessarily finite sets. For clones of operations this general characterization is mentioned 
in \cite{G}, proved in \cite{Po2} and \cite{Po3}, and it is implicit in \cite{Sz}.
 The characterization of 
closed sets of relations is given below 
in terms of certain closure conditions which are variants of those in 
\cite{G} and \cite{PK}, in the case of finite underlying sets, and in \cite{Sz},
 and in \cite{Po2} and \cite{Po3},
in the general case of arbitrary sets.

Recall that if $f$ is an $n$-ary $E$-valued function on $B$ and $g_1,\ldots ,g_n$ are all $m$-ary 
$B$-valued functions 
on $A$, then the composition $f(g_1,\ldots ,g_n)$ is an $m$-ary to $E$-valued function on $A$,
 and its value on ${\bf a}\in A^m$ is $f(g_1({\bf a}),\ldots ,g_n({\bf a}))$.
 In this section we are concerned with the special case $A=B=E$, and this set may be finite or infinite.

A \emph{clone} on $A$ is a set of operations $\mathcal{C}\subseteq \cup _{n\geq 1}A^{A^n}$
 such that it contains all projections (variables) and it is closed under composition.

An operation $f\in A^{A^n}$ \emph{preserves} a relation $R$ on $A$ if $fR\subseteq  R$, i.e.,
 if $f$ satisfies the constraint $(R,R)$. 
A class $\mathcal{F}\subseteq \cup _{n\geq 1}A^{A^n}$ is said 
to be \emph{definable} by a set $\mathcal{R}$ of relations of various arities on $A$, 
if $\mathcal{F}$ is the class of all 
operations which preserve every member of $\mathcal{R}$.   
Similarly, a set $\mathcal{R}$ of relations of various arities on $A$ is said 
to be \emph{characterized} by a set $\mathcal{F}$ of operations on $A$,
 if $\mathcal{R}$ is the set of all relations which are preserved by 
every member of $\mathcal{F}$.

\begin{theorem}\emph{\bf (P\" oschel)}
Let $A$ be an arbitrary non-empty set and let $\mathcal{C}$ be a set of operations on $A$. 
Then the following conditions are equivalent:
\begin{itemize}
\item[(i)] $\mathcal{C}$ is a locally closed clone;
\item[(ii)] $\mathcal{C}$ is definable by some set of relations of various arities on $A$.
\end{itemize}
\end{theorem}

\begin{proof} It is easy to see that $(ii)$ implies $(i)$.

To see the converse, assume $(i)$. According to Theorem 1, $\mathcal{C}$ 
is definable by some set $\mathcal{T}$ of $A$-to-$A$ constraints.
Consider any constraint $(R,S)$ in $\mathcal{T}$. Let $\bar {R}={\cup }_{f\in \mathcal{C}}fR$.
 Clearly $\bar {R}\subseteq S$.   
It follows that $\mathcal{C}$ is definable by $\{\bar {R}: (R,S)\in \mathcal{T}\}$. 
\end{proof}

We say that a set $\mathcal{R}$ of relations of various arities on $A$ is \emph{locally closed} if 
for every relation $R$ on $A$ the following holds: 
if for every finite subset $F$ of $R$ there is a relation $R'$ in $\mathcal{R}$ such that 
$F\subseteq R'\subseteq R$, then $R$ belongs to $\mathcal{R}$.

\begin{theorem}\emph{\bf (Szab\' o)}
Let $A$ be an arbitrary non-empty set and let $\mathcal{R}$ be a set of relations 
of various arities on $A$.
 Then the following are equivalent:
\begin{itemize}
\item[(i)]$\mathcal{R}$ is locally closed and contains the binary equality relation,
 the empty relation, and is closed under formation of tight conjunctive minors;
\item[(ii)] $\mathcal{R}$ is characterized by some set of operations on $A$. 
\end{itemize}
\end{theorem}

\begin{proof} It is not difficult to see that $(ii)$ implies $(i)$.

To prove the converse, define the set $\mathcal{T}$ of $A$-to-$A$ constraints by
 \begin{displaymath}
\mathcal{T}=\{(R,S): \textrm{ for every finite } F\subseteq R \textrm{ there is } R'\in \mathcal{R}
\textrm{ such that }F\subseteq R'\subseteq S \}.  
\end{displaymath} 
Note that $\mathcal{T}\supseteq \{(R,R): R\in \mathcal{R}\}$, and if $R\not\in \mathcal{R}$ 
then $(R,R)\not\in \mathcal{T}$.
By its definition $\mathcal{T}$ is locally closed.
Also $\mathcal{T}$ contains the binary equality and the empty constraints.

Let us show that $\mathcal{T}$ 
is closed under formation of conjunctive minors. For that, let $(R,S)$ be an $m$-ary 
conjunctive minor of the family
 $(R_j,S_j)_{j\in J}$ via a scheme  
$H=(h_j)_{j\in J}$, $h_j:n_j\rightarrow m\cup V$ where each $(R_j,S_j)$ is in $\mathcal{T}$.
 Let $F$ be a finite subset of $R$, with elements 
 ${\bf a}^1,\ldots ,{\bf a}^n$.
Since $R$ is a restrictive conjunctive minor of $(R_j)_{j\in J}$ via $H=(h_j)_{j\in J}$,
 there are Skolem maps ${\sigma }_i:V\rightarrow A$, $1\leq i\leq n$, 
such that, for every $j$ in $J$, the finite set $F_j$
whose elements are $({\bf a}^1+{\sigma }_1)h_j,\ldots ,({\bf a}^n+{\sigma }_n)h_j$ 
is contained in $R_j$.
By definition of $\mathcal{T}$, for every $j$ in $J$ there are relations $R^{'}_j$ in 
$\mathcal{R}$ such that $F_j\subseteq R^{'}_j\subseteq S_j$. 
Consider the tight conjunctive minor $R'$ of the family $( R^{'}_j)_{j\in J}$ via $H$. 
Since $S$ is an extensive conjunctive minor of $(S_j)_{j\in J}$ via $H=(h_j)_{j\in J}$,
 and, for every $j$ in $J$,  
$F_j\subseteq R^{'}_j\subseteq S_j$, it follows that $F\subseteq {R'}\subseteq S$.
In other words, $(R,S)$ belongs to $\mathcal{T}$.
So $\mathcal{T}$ is indeed closed under formation of conjunctive minors.

 According to Theorem 2, there 
is a set $\mathcal{F}$ of operations on $A$
which satisfy exactly those 
constraints that are in $\mathcal{T}$. Then $\mathcal{F}$ is a set of operations preserving
 exactly those relations which are in $\mathcal{R}$. 
\end{proof}

As mentioned above, in the general case of arbitrary 
underlying sets the conditions in $(i)$ of Theorem 5 characterizing the closed sets of relations
 are equivalent to those given by Szab\' o in \cite{Sz} and P\" oschel in \cite{Po2} and \cite{Po3}.
For the equivalence between Szab\' o's and P\" oschel's approaches see e.g. \cite{Po3}.

The following concept was introduced by P\" oschel in \cite{Po2}.
We shall again make use of the notation introduced in Section 3.  
Consider arbitrary non-empty sets $A$ and $B$.
Let $(R_j)_{j\in J}$ be a non-empty family of relations on $A$ where, for each $j\in J$, 
$R_j$ has arity $n_j$.
Let $m\geq 1$, ${\bf b}\in B^m$ and let $({\bf b}_j)_{j\in J}$ be a family with  
${\bf b}_j\in B^{n_j}$ for each ${j\in J}$.
 The $m$-ary relation $R$ on $A$ defined by
 \begin{displaymath}
R=\{ f{\bf b}\in A^m: f\in A^B \textrm{ and for each }j\in J, f{\bf b}_j\in R_j\}
\end{displaymath}
is said to be obtained from the family $(R_j)_{j\in J}$ by 
\emph{general superposition with respect to 
${\bf b}$ and the family $({\bf b}_j)_{j\in J}$}. 
(This reformulation appears in e.g. [P$\ddot o$4].)    
Recall that $f{\bf b}$ is the $m$-tuple $(f{\bf b}(i)\mid i\in m)$
and, for each ${j\in J}$, $f{\bf b}_j$ is the $n_j$-tuple $(f{\bf b}_j(i)\mid i\in n_j)$.
Note that general superposition subsumes formation of tight conjunctive minors of relations:
indeed if $R$ is a tight conjunctive minor of $(R_j)_{j\in J}$ via $H=(h_j)_{j\in J}$ then we can
define $B=m\cup V$, ${\bf b}={\iota }_{Bm}$
where $\iota _{Bm}$ is the canonical injection (inclusion map) from $m$ to $B$, and ${\bf b}_j=h_j$ where 
$h_j:n_j\rightarrow m\cup V$. 
Also, it is easy to see that the binary equality constraint can be obtained from the full unary relation
$R_0=A^1$ by general superposition with singleton $J$ and singleton $B$.

A set $\mathcal{R}$ of relations of various arities on $A$ is said to be \emph{closed under 
general superpositions} if whenever every member of a non-empty family $(R_j)_{j\in J}$ of relations
 is in $\mathcal{R}$, all relations obtained from the family $(R_j)_{j\in J}$ by general superpositions 
 are also in $\mathcal{T}$.

The following shows that the characterization of closed sets of relations
given in Theorem 5 is equivalent to that appearing in \cite{Po2}, \cite{Po3} and \cite{Po4}.

\begin{theorem}
Let $A$ be a non-empty set and $\mathcal{R}$ a locally closed set of relations of various arities on $A$
containig the empty relation. The following conditions are equivalent:
\begin{itemize}
\item[(i)]$\mathcal{R}$ contains the binary equality relation and
 is closed under formation of tight conjunctive minors;
\item[(ii)] $\mathcal{R}$ contains the full unary relation $A^1$ 
and is closed under general superpositions.
\end{itemize}
\end{theorem}

\begin{proof} The implication $(ii)\Rightarrow (i)$ follows from the observations above.

To prove that implication $(i)\Rightarrow (ii)$ holds note first that 
the full unary relation $A^1$ is a tight conjunctive minor of the binary equality relation.
Let us show that every relation obtained from a family of relations $(R_j)_{j\in J}$
 by general superposition can be obtained by intersecting 
a tight conjuctive minor of the family $(R_j)_{j\in J}$ with 
a tight conjunctive minor of the binary equality relation.

Let $(R_j)_{j\in J}$ be a non-empty family of relations on $A$ where, for each $j\in J$, 
$R_j$ has arity $n_j$, and  
let $R$ be the $m$-ary relation on $A$ obtained from the family $(R_j)_{j\in J}$ by 
general superposition with respect to 
${\bf b}\in B^m$ and family $({\bf b}_j)_{j\in J}$ with ${\bf b}_j\in B^{n_j}$ for each ${j\in J}$,
 where without loss of generality $B$ is a non-empty set disjoint from the ordinals.   
Consider the $m$-ary relation $R_{\bf b}^=$ on $A$ defined by
 \begin{displaymath}
 R_{\bf b}^==\{(a_t\mid t\in m ): a_i=a_j  \textrm{ for every $i,j\in m$ such that }{\bf b}(i)={\bf b}(j) \}
\end{displaymath} 
It is easy to see that $R_{\bf b}^=$ is a tight conjunctive minor of the binary equality relation.

Let $V$ be the complement in $B$ of the range of ${\bf b}$.
Consider the minor formation scheme $H=(h_j)_{j\in J}$
with target $m$, indeterminate set $V$ and source family 
$(n_j)_{j\in J}$, and where, for each $j\in J$, $h_j:n_j\rightarrow m\cup V$
is such that
\begin{displaymath}
({\bf b}+\iota _{BV})h_{j}={\bf b}_j
\end{displaymath} 
Consider the $m$-ary tight conjunctive minor $R'$ of the family $(R_j)_{j\in J}$ via $H$.

Let us show that $R=R'\cap R_{\bf b}^=$. Observe that $R\subseteq R_{\bf b}^=$.
 Let ${\bf a}\in R$. Then, for some function $f:B\rightarrow A$, ${\bf a}=f{\bf b}$ and, 
for each $j\in J$, $f{\bf b}_j\in R_j$. 
Define the Skolem map $\sigma :V\rightarrow A$ by $\sigma =f\iota _{BV}$.
By definition of $H$ and $\sigma $ we have 
\begin{displaymath}
({\bf a}+\sigma )h_j=(f{\bf b}+f\iota _{BV})h_{j}=[f({\bf b}+\iota _{BV})]h_{j}=f[({\bf b}+\iota _{BV})h_{j}]=f{\bf b}_j
\end{displaymath} 
for each $j\in J$. Thus, for every $j\in J$, $({\bf a}+\sigma )h_j\in R_j$, and we have ${\bf a}\in R'$.
Since $R\subseteq R_{\bf b}^=$, we conclude $R\subseteq R'\cap R_{\bf b}^=$.

To show that $R'\cap R_{\bf b}^=\subseteq R$, let ${\bf a}\in R'\cap R_{\bf b}^=$.
Since $R'$ is the tight conjunctive minor of the family $(R_j)_{j\in J}$ via $H$,   
there is a Skolem map $\sigma :V\rightarrow A$ such that for every $j$ in $J$ we have 
$({\bf a}+\sigma )h_j\in R_j$. It follows from definition of $R_{\bf b}^=$ that, for every
$i$ and $j$ in $J$, ${\bf a}(i)={\bf a}(j)$ if ${\bf b}(i)={\bf b}(j)$. 
Let $f:B\rightarrow A$ be such that $f{\bf b}={\bf a}$ and $f\iota _{BV}=\sigma $.
We have 
\begin{displaymath}
f{\bf b}_j=f[({\bf b}+\iota _{BV})h_{j}]=[f({\bf b}+\iota _{BV})]h_{j}=(f{\bf b}+f\iota _{BV})h_{j}=({\bf a}+\sigma )h_j
\end{displaymath} 
and so, for each $j\in J$, $f{\bf b}_j\in R_j$. Thus $f{\bf b}\in R$, that is, ${\bf a}\in R$.
\end{proof}

\end{document}